\documentclass[12pt,a4paper,reqno,intlimits]{amsart}
\usepackage{amsfonts, amsmath, amsxtra, amsthm, amssymb, latexsym, mathrsfs, srcltx, dsfont, euscript}

\usepackage[dvips]{graphicx}

\usepackage{hyperref}
\hypersetup
{  pdftitle={Shown in AR File Information},
   pdfstartview=FitH,   
   bookmarks=true,      
}

\theoremstyle{plain}

\newtheorem*{theorem*}{Theorem}
\newtheorem*{corollary*}{Corollary}
\newtheorem*{lemma*}{Lemma}
\newtheorem*{property*}{Property}
\newtheorem{proposition*}{Proposition}
\newtheorem*{statement*}{Statement}

\newtheorem{lemma}{Lemma}

\newtheorem{proposition}{Proposition}

\theoremstyle{definition}

\newtheorem*{definition*}{Definition}
\newtheorem*{remark*}{Remark}
\newtheorem*{example*}{Example}

\newtheorem{example}{Example}

\allowdisplaybreaks

\makeatletter
\renewcommand{\@biblabel}[1]{#1.\hfill}
\makeatother

\renewcommand{\title}[1]{
\begin{center}
{\bfseries #1}
\end{center}}

\renewcommand{\author}[1]
{\begin{center} {\bfseries{\copyright~~~#1}}
\end{center}}

\renewcommand{\address}[1]
{\begin{center}
\tiny
{\scshape{#1}}
\end{center}\medskip}

\begin{document}

\pagenumbering{arabic}

\thispagestyle{empty}

\markboth{{\footnotesize{\it \textbf{I.\,I.~Karpenko, D.\,L.~Tyshkevich}}}} {{\footnotesize {\it \textbf{1--D Schr\"odinger Operators With $\delta$--interactions}}}}

\medskip
\normalsize

\title{ON SELF-ADJOINTNESS OF 1--D SCHR\"ODINGER OPERATORS WITH $\delta$--INTERACTIONS}

\author{I.\,I.~Karpenko}
\address{Tavrida National V.\,I.~Vernadsky University, 4 Academician Vernadsky Ave, Simferopol, 95007, Ukraine}
\email{i\_karpenko@ukr.net}

\author{D.\,L.~Tyshkevich}
\address{Tavrida National V.\,I.~Vernadsky University, 4 Academician Vernadsky Ave, Simferopol, 95007, Ukraine}
\email{dtyshk@inbox.ru}

\subsection*{Abstract}
In the present work we consider in $L^2(\mathbb{R}_+)$ the Schr\"odinger operator
$\mathrm{H_{X,\alpha}}=-\mathrm{\frac{d^2}{dx^2}}+\sum_{n=1}^{\infty}\alpha_n\delta(x-x_n)$.
We investigate and complete the conditions of self-adjointness and
nontriviality of deficiency indices for $\mathrm{H_{X,\alpha}}$
obtained in \cite{karpiiKost}. We generalize  the conditions
found earlier in the special case $d_n:=x_{n}-x_{n-1}=1/n$, $n\in
\mathbb{N}$, to a wider class of sequences $\{x_n\}_{n=1}^\infty$.
Namely, for $x_n=\frac{1}{n^{\gamma}\ln^\eta n}$ with
$\langle\gamma,\eta\rangle\in(1/2,\,1)\!\times\!(-\infty,+\infty)\:\cup\:\{1\}\!\times\!(-\infty,1]$,
the description of asymptotic behavior of the sequence
$\{\alpha_n\}_{n=1}^{\infty}$ is obtained for
$\mathrm{H_{X,\alpha}}$ either to be self--adjoint or to have
nontrivial deficiency indices.

\subsection*{Keywords}
Schr\"odinger operator, local point interaction,
self-adjoint operator, deficiency indices, Jacobi matrices

\subsection*{2000 Mathematics Subject Classification (MSC2000)}
34L40, 47E05, 47B25, 47B36, 81Q10

\section{Introduction}

Let $X=\{x_n\}_{n=0}^{\infty}$ be a strictly increasing sequence of nonnegative numbers,
$x_0=0$, and $\lim_{n\rightarrow\infty}x_n=\infty.$ Let also
$\alpha=\{\alpha_n\}_{1}^{\infty}$ be a sequence of real numbers.

The differential expression
\begin{equation}\label{EqlXalpha}
   l_{X,\alpha}:=-\mathrm{\frac{d^2}{dx^2}}+\sum_{n=1}^{\infty}\alpha_n\delta(x-x_n)
\end{equation}
on $L^2(0,+\infty)$ is connected with the symmetric differential operator
\begin{equation}\label{EqH0Xalpha}
  \mathrm{H^0_{X,\alpha}}:=-\mathrm{\frac{d^2}{dx^2}}
\end{equation}
with domain
\begin{equation}\label{EqdomH0Xalpha}
   \begin{split}
      \mathrm{dom(H^0_{X,\alpha})}=
          \big\{f\in W^{2,2}(\mathbb{R}_+\setminus X)\cap L^2_{\mathrm{comp}}(\mathbb{R}_+)\ |\
      &         f'(0)=0,\:
                \\
      &         f'(x_n{\scriptscriptstyle+})-f'(x_n{\scriptscriptstyle-})=\alpha_n f(x_n)
          \big\}.
  \end{split}
\end{equation}
Denote by $\mathrm{H_{X,\alpha}}$ the closure of the operator
$\mathrm{H^0_{X,\alpha}}$.

Schr\"odinger operators  with distributional potentials  have
attracted considerable interest in the last decades, in
particular, because they can be used as solvable models in many
situations, see \cite{karpiiAlb_Ges88}~--- \cite{karpiiBra85},
\cite{karpiiBusch9}, \cite{karpiiGesz6}, \cite{karpiiOgur},
\cite{karpiiShub45}. For instance, the operator
$\mathrm{H^0_{X,\alpha}}$ can be regarded as a Hamiltonian for a
$\delta$--interaction at points $x_n$ with intensity $\alpha_n$.
In the general case, the operator $\mathrm{H_{X,\alpha}}$ does not
need to be self-adjoint. One of important problems in the spectral
analysis of this operator is to find necessary and sufficient
conditions for the operator  $\mathrm{H_{X,\alpha}}$  to be
self-adjoint. The affirmative answer to this problem has recently
been obtained  in  the case of lower semi-bounded Hamiltonians.
Namely, it is proved in \cite{karpiiAKM10} (see laso
\cite{karpiiHryn}) that $\mathrm{H_{X,\alpha}}$ is always
self-adjoint provided that it is lower semi-bounded.

Spectral properties of the operator $\mathrm{H_{X,\alpha}}$ depend
on both the sequence $\alpha$ and the sequence $X$. In the latter
case, the behavior of the sequence
\begin{equation}\label{EqDefdn}
   d_n:=x_n-x_{n-1},\qquad n\in\mathbb{N},
\end{equation}
is an important characteristic. In particular, if
$d_{*}:=\mathrm{inf}_{n\in \mathbb{N}}d_n>0,$ then the operator
$\mathrm{H_{X,\alpha}}$ is \textit{always self-adjoint}
\cite{karpiiGesz6}.

This result is sharp  in the sense that there is a sequence $d_n$
satisfying $\lim_{n\to\infty}d_n=0$ such that the Hamiltonian
$\mathrm{H_{X,\alpha}}$ has nontrivial deficiency indices for some
sequences $\alpha\subset \mathbb{R}$. Namely, C.~Shubin~Christ and
G.~Stolz have shown in \cite{karpiiShub45} that
$n_{\pm}(\mathrm{H_{X, \alpha}})= 1$ if $d_n=1/n$ and
$\alpha_n=-2n-1$, $n\in \mathbb{N}$. Thus the case $d_{*}=0$ is
fundamentally different from the case $d_{*}>0$ since nontrivial
deficiency indices can be realized there.

In \cite{karpiiKost}, A.\,S.~Kostenko and M.\,M.~Malamud studied
the Hamiltonian  $\mathrm{H_{X,\alpha}}$ in the framework of
boundary triplets and the corresponding Weyl function. Such an
approach to the theory of extensions of symmetric operators was
initiated about thirty years ago and still continues to actively
develop, see \cite{karpiiBruen8}~--- \cite{karpiiGorb20}.

Using a corresponding boundary triple, the authors in
\cite{karpiiKost} have parametrized the set of Hamiltonians
$\mathrm{H_{X,\alpha}}$ with certain classes of Jacobi matrices
(three-diagonal matrices). It was also found there that spectral
properties of the Hamiltonian  $\mathrm{H_{X,\alpha}}$ are closely
linked with the same properties of the corresponding Jacobi
matrix,
\begin{equation}\label{EqJacobyMatr}
   B_{X,\alpha}=
   \begin{pmatrix}
      r_1^{-2}(\alpha_1+\frac{1}{d_1}+\frac{1}{d_2}) & -r_1^{-1}r_2^{-1}d_2^{-1}
                                                              & 0                                              & \ldots  \\
     -r_1^{-1}r_2^{-1}d_2^{-1}                       & r_2^{-2}(\alpha_2+\frac{1}{d_2}+\frac{1}{d_3})
                                                              & -r_2^{-1}r_3^{-1}d_3^{-1}                      & \ldots  \\
       0                                             & -r_2^{-1}r_3^{-1}d_3^{-1}
                                                              & r_3^{-2}(\alpha_3+\frac{1}{d_3}+\frac{1}{d_4}) & \ldots \\
      \ldots                                         & \ldots
                                                              & \ldots                                         & \ldots
   \end{pmatrix},
\end{equation}
where
\begin{equation}\label{EqDefrn}
   r_n=\sqrt{d_n+d_{n+1}},\qquad n\in \mathbb{N}.
\end{equation}
As it turned out the deficiency indices for
$\mathrm{H_{X,\alpha}}$ and $B_{X,\alpha}$ coincide,
$n_{\pm}(\mathrm{H_{X,\alpha}})=n_{\pm}( B_{X,\alpha})$
(\cite[Theorem 5.4]{karpiiKost}) and, consequently,
$n_{\pm}(\mathrm{H_{X,\alpha}})\leq
 1$, see \cite{karpiiBusch9}.
In particular, $\mathrm{H_{X,\alpha}}$ is self-adjoint if and only if the matrix
$B_{X,\alpha}$ is self-adjoint.

Using the Carleman criterion A.\,S.~Kostenko and M.\,M.~Malamud have
obtained the following result \cite[Proposition 5.7]{karpiiKost}:

\medskip
\begin{center}\
\parbox{0.95\textwidth}
{
\textit{if $\sum_{n=1}^{\infty}d_n^2=\infty$, then the operator
$\mathrm{H_{X,\alpha}}$ is self-adjoint for any sequence $\alpha\subset \mathbb{R}$.}
}
\end{center}

In comparison with the  mentioned result from \cite{karpiiAKM10},
\cite{karpiiHryn}, this statement from \cite{karpiiKost} gives
new information only for not lower semi-bounded Hamiltonians
$\mathrm{H_{X,\alpha}}$.

\medskip
\noindent Clearly, the result on self-adjointness of the operator
$\mathrm{H_{X,\alpha}}$ for $d_{*}:=\mathrm{inf}_{n\in
\mathbb{N}}d_n>0$ is a particular case of the latter statement.

Moreover, the example of Shubin Christ and Stolz in
\cite{karpiiShub45} was significantly specified in
\cite{karpiiKost}, namely,

\medskip
\begin{center}\
\parbox{0.9\textwidth}
{
\textit{if $d_n=1/n$, then:}\\
\textbf{1.}\;%
\textit{$n_{\pm}(\mathrm{H_{X,\alpha}})= 0$ if
$\alpha_n\leq-2(2n+1)+\frac{C_1}{n},\ C_1>0,$ or $\alpha_n\geq
-\frac{C_2}{n},\ C_2>0$};\label{Pagesadj1n}

\textbf{2.}\;%
\textit{$n_{\pm}(\mathrm{H_{X,\alpha}})=1$ if $\alpha_n=a(2n+1)+O(1/n),\ a\in (-2,0)$.}
}
\end{center}

\medskip
\noindent Note that the estimates in \textbf{1} follow from the
estimates in \textbf{(ii)}, \textbf{(iii)} that were obtained for
a more general case of $\{d_n\}_1^{\infty}\in \ell_2\setminus
\ell_1$ in \cite[Proposition 5.11]{karpiiKost} with the use of
sufficient conditions for a Jacobi matrix to be self-adjoint.
\begin{proposition}\label{PropMalKost511}\cite[Prop.\,5.11]{karpiiKost}
The operator $\mathrm{H_{X,\alpha}}$ is self-adjoint on
$L^2(\mathbb{R}_+)$ if the sequence
$\alpha=\{\alpha_n\}_1^{\infty}$ and the sequences
$\{d_n\}_1^{\infty},\ \{r_n\}_1^{\infty}$ defined by
\eqref{EqDefdn} and \eqref{EqDefrn}, correspondingly, satisfy one
of the following conditions:
\begin{enumerate}
   \item[\textbf{(i)}] $\sum_{n=1}^{\infty}|\alpha_n| d_nd_{n+1}r_{n-1}r_{n+1}=\infty$;
   \item[\textbf{(ii)}] there exists a constant $C_1>0$ such that
         \begin{equation*}
            \alpha_n+\frac{1}{d_n}\Bigl(1+\frac{r_n}{r_{n-1}}\Bigr)+
            \frac{1}{d_{n+1}}\Bigl(1+\frac{r_n}{r_{n+1}}\Bigr)\leq
            C_1(d_n+d_{n+1}), \qquad n\in \mathbb{N};
         \end{equation*}
   \item[\textbf{(iii)}] there exists a constant $C_2>0$ such that
      \begin{equation*}
          \alpha_n+\frac{1}{d_n}\Bigl(1-\frac{r_n}{r_{n-1}}\Bigr)+
          \frac{1}{d_{n+1}}\Bigl(1-\frac{r_n}{r_{n+1}}\Bigr)\geq
          -C_2(d_n+d_{n+1}), \qquad n\in \mathbb{N}.
          \end{equation*}
\end{enumerate}
\end{proposition}

In this paper, we continue the study of the conditions obtained in
\cite{karpiiKost} for the Schr\"odinger operator
$\mathrm{H_{X,\alpha}}$ to be self-adjoint or to have nontrivial
deficiency indices. It turned out that these conditions found for
$d_n=1/n$ can be generalized to a broader class of sequences, see
Propositions~\ref{PropAsympt}, \ref{PropGendDiffable},
\ref{PropGendDiffable2}. For example, for a class of the sequences
$\Big\{\dfrac{1}{n^{\gamma}\ln^\eta n}\Big\}_2^{\infty}$ that
belong to $\ell_2\setminus \ell_1$ if
$\langle\gamma,\eta\rangle\in
(1/2,\,1)\!\times\!(-\infty,+\infty)\:\cup\:\{1\}\!\times\!(-\infty,1],$
we obtain a description for the asymptotic behavior of the
sequence $\alpha$ such that the operator $\mathrm{H_{X,\alpha}}$
would either be self-adjoint or have nontrivial deficiency
indices.

 \section{Sufficient conditions for self-adjointness of the operator $\mathrm{H_{X,\alpha}}$}

Taking into account the above considerations, we will mention some
properties of sequences $\{d_n\}_1^{\infty}\in \ell_2\setminus
\ell_1$ of positive numbers. The most important of them is the
property:
\begin{equation}\label{EqFromDalamb}
   \underset{n\rightarrow\infty}{\lim\ \inf}\frac{d_{n+1}}{d_n}
   \leq 1\leq
   \underset{n\rightarrow\infty}{\lim\ \sup}\frac{d_{n+1}}{d_n},
\end{equation}
which follows from the d'Alembert test for series.
This immediately implies that if there exists the
$\lim_{n\rightarrow\infty}\frac{d_{n+1}}{d_n}$
for a sequence $\{d_n\}_1^{\infty}\in \ell_2\setminus \ell_1$ of positive numbers,
then this limit equals 1. Using \eqref{EqFromDalamb} and making certain
restrictions on the sequence $\{d_n\}_1^{\infty}$ we can significantly simplify
the sufficient conditions of self-adjointness in \textbf{(i)}~---
\textbf{(iii)}.

\begin{proposition}\label{Propdn3}
Given a sequence $\alpha=\{\alpha_n\}_1^{\infty}$. Let also the
sequence $\{d_n\}_1^{\infty}\in \ell_2\setminus \ell_1$ of
positive numbers defined by \eqref{EqDefdn} satisfy the
relation:
\begin{equation}\label{EqdnRestrns}
    \underset{n\rightarrow\infty}{\lim\ \inf}\frac{d_{n+1}}{d_n}>0.
\end{equation}
Then $\mathrm{H}_{X,\alpha}$ is self-adjoint provided that there
holds the following condition:
\begin{enumerate}
\item[\textbf{(I)}]\;
   $\sum_{n=1}^{\infty}|\alpha_n|d_n^3=\infty.$
\end{enumerate}
\end{proposition}

\begin{proof}[Proof]
In fact, condition \eqref{EqdnRestrns} implies that there exists a $C>0$ such that
$$
d_{n+1}>Cd_n,\qquad n\in \mathbb{N}.
$$
It follows that
\begin{equation}\label{EqPropdn32}
\begin{split}
  & d_nd_{n+1}r_{n-1}r_{n+1}=
    d_nd_{n+1}\sqrt{d_{n-1}+d_n}\sqrt{d_{n+1}+d_{n+2}}>\\
  & d_n(Cd_n)\sqrt{d_n}\
  \sqrt{Cd_{n}+C^2d_{n}}>C\sqrt{C+C^2}d_n^3,
   \qquad n\in\mathbb{N}.
\end{split}
\end{equation}
Thus, the divergence of series \textbf{(I)} implies the divergence
of series \textbf{(i)} in Proposition~\ref{PropMalKost511}, and
there holds the sufficient condition for $\mathrm{H}_{X,\alpha}$
to be self-adjoint.
\end{proof}

Note that, for the class of sequences $\{d_n\}_1^{\infty}\in
\ell_2\setminus \ell_1$ satisfying the additional constraint
\begin{equation}\label{EqdnRestrns_1}
   0<\underset{n\rightarrow\infty}{\lim\ \inf}\frac{d_{n+1}}{d_n}\leq
   \underset{n\rightarrow\infty}{\lim\ \sup}\frac{d_{n+1}}{d_n}<\infty,
\end{equation}
both series \textbf{(I)} and \textbf{(i)} converge and diverge
simultaneously. Hence test \textbf{(i)} as well as
Proposition~\ref{Propdn3} can be applied for such sequences.

In the next assertion, we present the conditions sufficient for
tests \textbf{(ii)}, \textbf{(iii)} of
Proposition~\ref{PropMalKost511} to hold (for now, without any
additional restrictions on the sequence $\{d_n\}_1^{\infty}$).
These conditions will allow us to find simpler sufficient
conditions for the Hamiltonian $\mathrm{H}_{X,\alpha}$ to be
self-adjoint.

\begin{proposition} \label{PropFGgen}
Let sequences $\{d_n\}_1^{\infty}\in \ell_2\setminus \ell_1$ and
$\{r_n\}_1^{\infty}$ be defined by \eqref{EqDefdn}, \eqref{EqDefrn}, and let the function
\begin{equation}\label{EqFDef}
  F(n)=\frac{1}{d_{n}}\Bigl(\frac{r_n}{r_{n-1}}-1\Bigr)
  +\frac{1}{d_{n+1}}\Bigl(\frac{r_n}{r_{n+1}}-1\Bigr)
  \quad(n\in\mathbb{N})
\end{equation}
allow the representation of the form:
\begin{equation}\label{EqFLessG}
    F(n)= G(n)+O(d_n)
   \quad(n\in\mathbb{N})
\end{equation}
(for definiteness, we put $r_0:=1$). Then the Hamiltonian
$\mathrm{H}_{X,\alpha}$ is self-adjoint provided that
one of the following conditions hold%
\footnote{ As we can see from the proof of
Proposition~\ref{PropFGgen}, under $O(d_n)$ in \eqref{EqFLessG}
and \textbf{(II)} can be meant different sequences.}:
\begin{enumerate}
\item[\textbf{(II)}]\;there exists a constant $C_1>0$ such that
$$
\alpha_n\leq-\Bigl(\dfrac{2}{d_n}+\dfrac{2}{d_{n+1}}+
                       G(n)\Bigr)+C_1d_n\quad(n\in\mathbb{N});
$$
\item[\textbf{(III)}]\; there exists a constant $C_2>0$ such that
                       $\alpha_n\geq G(n)-C_2d_n\quad(n\in\mathbb{N})$.
\end{enumerate}
\end{proposition}

\begin{proof}[Proof]
By condition \textbf{(II)}, we have:
$$
\alpha_n+\dfrac{2}{d_n}+\dfrac{2}{d_{n+1}}+
                       G(n)\leq C_1d_n\quad(n\in\mathbb{N}).
$$
Then
$$
\alpha_n+\dfrac{2}{d_n}+\dfrac{2}{d_{n+1}}+ F(n)-O(d_n)\leq
                       C_1d_n\quad(n\in\mathbb{N}).
$$
Since the sequence $\{d_n\}_1^\infty$ is positive, we conclude that
$$
\alpha_n+\frac{1}{d_n}\Bigl(1+\frac{r_n}{r_{n-1}}\Bigr)+
            \frac{1}{d_{n+1}}\Bigl(1+\frac{r_n}{r_{n+1}}\Bigr)\leq
            Cd_n\leq C(d_n+d_{n+1})\quad(n\in \mathbb{N})
$$
for some $C>0.$ Consequently, it follows from the sufficient
condition \textbf{(ii)} of Proposition~\ref{PropMalKost511} that
the Hamiltonian $\mathrm{H}_{X,\alpha}$ is self-adjoint.

Arguing similarly we can prove that condition \textbf{(III)} implies estimate
\textbf{(iii)} of Proposition~\ref{PropMalKost511}.
\end{proof}

It is important to note that, for sequences $\{d_n\}_1^{\infty}\in
\ell_2\setminus \ell_1$ satisfying \eqref{EqdnRestrns_1}, both
estimates \textbf{(ii)}, \textbf{(II)} and
\textbf{(iii)},\textbf{(III)} are fulfilled or not fulfilled
simultaneously. Hence Proposition~\ref{PropMalKost511} and
Proposition~\ref{PropFGgen} are equivalent for such sequences.
Note also that tests \textbf{(II)} and \textbf{(III)} are of
common use only in the case when the function $G$ in decomposition
\eqref{EqFLessG} has simpler form than the function $F$. In this
case, there are of great interest sequences $\{d_n\}_1^\infty$
such that $G$ can be chosen as
\textit{zero} function\label{PagefootsupFndn}%
\footnote{See Proposition~\ref{PropGendDiffable} and
Example~\ref{Exngammalneta} in what follows. Relation
\eqref{EqFLessG} implies directly that $G$ is zero function if and
only if $\sup_{n\in\mathbb{N}}\frac{F(n)}{d_n}<\infty$.
}%
\newcounter{footsupFndn}%
\setcounter{footsupFndn}{\value{footnote}}%
: due to test \textbf{(III)}, all the Hamiltonians
$\mathrm{H}_{X,\alpha}$ with nonnegative sequences $\alpha$ are
self-adjoint. In the propositions below, we present a number of
requirements to properties of a sequence $\{d_n\}_1^\infty$ in
order to provide the asymptotics $F(n)=O(d_n)$.

\begin{proposition}\label{PropAsympt}
Let a sequence $\{d_n\}_1^{\infty}\in
\ell_2\setminus \ell_1$ satisfy the following asymptotic estimate:
\begin{equation}\label{asympt}
\frac{d_{n+1}}{d_n}=1+Cd_n+O(d_n^2).
\end{equation}
Then $F(n)\ =\ O(d_n).$
\end{proposition}

\begin{proof}[Proof]
Using~\eqref{asympt} and carrying out direct calculations
we can show that the following relations hold:
\begin{enumerate}
  \item[\textbf{(a0)}]\;$\frac{d_n}{d_{n+1}}=1-Cd_n+O(d_n^2 )$.
  \item[\textbf{(a1)}]\;$\frac{r_n}{r_{n-1}}=\sqrt{\frac{d_n+d_{n+1}}{d_{n-1}+d_{n}}}=
                        \sqrt{\frac{1+d_{n+1}/d_n}{1+d_{n-1}/d_n}}=\sqrt{\frac{2+Cd_n+O(d_n^2)}{2-Cd_n+O(d_n)}}
                        =1+Cd_n+O(d_n^2);$
  \item[\textbf{(a2)}]\; Similarly, $\frac{r_n}{r_{n+1}}=1-Cd_n+O(d_n^2).$
  \end{enumerate}

In this case, we obtain the following relations:
\begin{equation*}
  \begin{split}
     & F(n)=\tfrac{1}{d_n}\big(\tfrac{r_n}{r_{n-1}}-1\big) +
                          \tfrac{1}{d_{n+1}}
                          \big(\tfrac{r_n}{r_{n+1}}-1\big)
                          =\tfrac{1}{d_n}\big(\tfrac{r_n}{r_{n-1}}-1
                           +\tfrac{d_n}{d_{n+1}}\big(\tfrac{r_{n}}{r_{n+1}}-1 \big) \big)
                           \overset{\textbf{a0--a2}}{=}\\
&    =\tfrac{1}{d_n}
           \big(Cd_n+O(d_n^2)+\Big(1-Cd_n+O(d_n^2 )\big)\big(-Cd_n+O(d_n^2)\big)
           \big)= \\
 &    =\tfrac{1}{d_n}
           \big(Cd_n+O(d_n^2)-Cd_n+O(d_n^2)\big)
   =\tfrac{1}{d_n}O(d_n^2)=O(d_n).
     \end{split}
\end{equation*}
\end{proof}

\begin{example}\label{Exngamma}
{\em Suppose that the Hamiltonian $\mathrm{H_{X,\alpha}}$ is
generated by the differential expression \eqref{EqlXalpha} as
explained in Introduction, with  $X=\{x_n\}_{0}^{\infty}$ defined
by the relations:
\begin{align}
    & x_0=0,\quad x_n=x_{n-1}+d_n\quad(n\in \mathbb{N}),\notag \\
    &  d_n=\dfrac{1}{n^{\gamma}},\ \gamma\in(1/2,\,1).\label{EqExngamma1}
    \end{align}

Then $\mathrm{H}_{X,\alpha}$ is self-adjoint provided that there holds one of
the following conditions:
\begin{enumerate}
   \item[\textbf{(s1)}]\;$\sum_{n=1}^{\infty}|\alpha_n|n^{-3\gamma}=\infty\,$.
   \item[\textbf{(s2)}]\; There exists a constant $C_1>0$ such that
   $$
   \alpha_n\leq-2\big(n^{\gamma}+(n+1)^{\gamma}\big)+
                           \frac{C_1}{n^{\gamma}}\quad(n\in \mathbb{N}).
   $$\label{PageExngamma}

   \item[\textbf{(s3)}]\;There exists a constant $C_2>0$ such that
    \quad $\alpha_n\geq -C_2\frac{1}{n^{\gamma}}\quad(n\in \mathbb{N})$.
\end{enumerate}
} Indeed, in this case we have $\{d_n\}_1^{\infty}\in
\ell_2\setminus \ell_1$ and, by direct calculations, we conclude
that
$\frac{d_{n+1}}{d_n}=1-\frac{1}{n^{\gamma}}+O(\frac{1}{n^{2\gamma}})$.
Then, in view of Proposition~\ref{Propdn3}, condition
\textbf{(s1)} provides the self-adjointness of the operator
$\mathrm{H}_{X,\alpha}$. To prove that \textbf{(s2)} and
\textbf{(s3)} can be applied, it suffices to note that the
sequence $\{d_n\}_1^{\infty}$ satisfy the assumptions of
Proposition~\ref{PropAsympt} and, therefore, $G(n)=0$. Further, we
can use conditions \textbf{(II)}, \textbf{(III)} of
Proposition~\ref{PropFGgen} directly.

As was mentioned before, the particular case $\gamma=1$ was
considered in \cite{karpiiKost} (see \textbf{1} on
p.\,\pageref{Pagesadj1n}).
\end{example}

In some cases (see Example~\ref{Exngammalneta})
condition~\eqref{asympt} is too strict. Nevertheless,
if~\eqref{asympt} does not hold, it is possible to carry out more refined
analysis of properties of the sequence $\{d_n\}_1^{\infty}$
leading to the asymptotics $F(n)\ =\ O(d_n).$

\begin{proposition} \label{PropGendDiffable}
Let $p\in\mathbb{N}$, and let a sequence $\{d_n\}_1^{\infty}\in
\ell_2\setminus \ell_1$ be generated by the function
$d\colon\;d_n:=d(n)\;\:(n\in\overline{p,\infty})$ defined on the
interval $(p,\infty)$ and twice continuously differentiable on it.
Let also the function $d$ satisfy the following conditions:
\begin{enumerate}
  \item[\textbf{(d0)}]\;$\frac{d_{n+1}}{d_n}=1+O(d_n)$.
  \item[\textbf{(d1)}]\;$d'(n)\neq0\;\:(n\in\overline{m+1,\infty})$\quad{and}\quad%
                        $\sup_{\substack{n\in\overline{m+1,\infty} \\
                                         \zeta,\eta\in[-1,2]
                                        }
                              }
                         \frac{|d'(n+\zeta)|}{|d'(n+\eta)|}<\infty
                        $ for some number $m\in\overline{p,\infty}$.
  \item[\textbf{(d2)}]\;$d''(n)\neq0\;\:(n\in\overline{m+1,\infty})$\quad{and}\quad%
                        $\sup_{\substack{n\in\overline{m+1,\infty} \\
                                         \zeta,\eta\in[-1,2]
                                        }
                              }
                         \frac{|d''(n+\zeta)|}{|d''(n+\eta)|}<\infty
                        $ for some number $m\in\overline{p,\infty}$.
  \item[\textbf{(d3)}]\;$\frac{d''(n)}{d'(n)}=O(d_n).$
  \end{enumerate}
Then $F(n)\ =\ O(d_n).$
\end{proposition}

\begin{proof}[Proof]
Since $d_n\rightarrow0,\ n\rightarrow\infty$, for arbitrary positive numbers
$u,\ v$ the following relations hold:
\begin{equation}\label{EqPropVelocConverg2.1}
    \mathbf{a)}\;\:
    \big(u+O(d_n)\big)^{-1}=
    u^{-1}+O(d_n);\qquad
    \mathbf{b)}\;\:
    \sqrt{\frac{u+O(d_n)}{v+O(d_n)}}=
    \sqrt{\frac{u}{v}}+O(d_n).
\end{equation}

Then condition $\textbf{(d0)}$ implies the preliminary asymptotic estimates:
\begin{enumerate}
  \item[\textbf{(b0)}]\;$\frac{d_n}{d_{n+1}}=1+O(d_n);$
  \item[\textbf{(b1)}]\;$\frac{r_n}{r_{n-1}}=1+O(d_n);$
  \item[\textbf{(b2)}]\;$\frac{r_n}{r_{n+1}}=1+O(d_n).$
  \end{enumerate}

In this case, the following relations hold:
\begin{equation}\label{EqFn}
  \begin{split}
     & F(n)=\tfrac{1}{d_n}\big(\tfrac{r_n}{r_{n-1}}-1\big) +
                          \tfrac{1}{d_{n+1}}
                          \big(\tfrac{r_n}{r_{n+1}}-1\big)
           =\tfrac{1}{d_n}\big( \tfrac{r_n}{r_{n-1}}-1
                          \big)
                          \big(1+\tfrac{d_n}{d_{n+1}}\cdot\tfrac{r_{n-1}}{r_{n+1}}\cdot
                              \tfrac{r_n-r_{n+1}}{r_n-r_{n-1}}
                          \big)
                          \overset{\textbf{(b1)}}{=}\\
     &    =\tfrac{1}{d_n}O(d_n)
           \Big( 1+\tfrac{d_n}{d_{n+1}}\cdot
                   \tfrac{r_{n-1}}{r_{n+1}}\cdot
                   \tfrac{r_n+r_{n-1}}{r_n+r_{n+1}}\cdot
                   \tfrac{r_n^2-r_{n+1}^2}{r_n^2-r_{n-1}^2}
           \Big)= \\
     &    =\tfrac{1}{d_n}O(d_n)
           \Big( 1+\tfrac{d_n}{d_{n+1}}\cdot
                   \big(\tfrac{r_{n-1}}{r_{n+1}}
                   \big)^2\cdot
                   \tfrac{r_n/r_{n-1}+1}{r_n/r_{n+1}+1}\cdot
                   \tfrac{r_n^2-r_{n+1}^2}{r_n^2-r_{n-1}^2}
           \Big)
           \overset{(\textbf{b0--b2},\ref{EqPropVelocConverg2.1},\ref{EqDefrn})}{=}\\
     &    =\tfrac{1}{d_n}O(d_n)
           \big(1+\big(1+O(d_n)\big)
                  \big(1+O(d_n)\big)^2\cdot
                  \tfrac{2+O(d_n)}{2+O(d_n)}\cdot
                  \tfrac{d_n-d_{n+2}}{d_{n+1}-d_{n-1}}
           \big)
           \overset{\eqref{EqPropVelocConverg2.1}}{=}\\
     &    =\tfrac{1}{d_n}O(d_n)
           \big(1+\big(1+O(d_n)\big)
                  \tfrac{d_n-d_{n+2}}{d_{n+1}-d_{n-1}}
           \big)=\\
     &    =\tfrac{1}{d_n}O(d_n)
           \big(\tfrac{(d_{n+1}-d_{n-1})-(d_{n+2}-d_n)}{d_{n+1}-d_{n-1}}+
                O(d_n)\tfrac{d_n-d_{n+2}}{d_{n+1}-d_{n-1}}
           \big).
  \end{split}
\end{equation}

Below we will need properties \textbf{(d1--d3)} of the sequence
$\{d_n\}_1^\infty$. Namely, for a sufficiently large $n$  and for some $C>0$
we obtain:
\begin{equation*}
   \begin{split}
      & \frac{(d_{n+1}-d_{n-1})-(d_{n+2}-d_n)}{d_{n+1}-d_{n-1}}=
        2\frac{d'(n+\zeta_n)-d'(n+\theta_n)}{d'(n+\zeta_n)}=\\
      & =2(\zeta_n-\theta_n)\frac{d''(n+\xi_n)}{d'(n+\zeta_n)}=
         2(\zeta_n-\theta_n)\frac{d''(n+\xi_n)}{d''(n)}\cdot
         \frac{d'(n)}{d'(n+\zeta_n)}\cdot\frac{d''(n)}{d'(n)}
         \overset{(\textbf{d1--d3})}=O(d_n);\\
      & \Big|\frac{d_n-d_{n+2}}{d_{n+1}-d_{n-1}}\Big|=
        \Big|\frac{d'(n+\theta_n)}{d'(n+\zeta_n)}\Big|\overset{\textbf{(d1)}}<C,\\
  \end{split}
\end{equation*}
where $\zeta_n\in [-1,1],\ \theta_n\in[0,2],\ \xi_n\in[-1,2]$. Then \eqref{EqFn}
implies the following estimate of the function $F(n)$:
$$
F(n)=\frac{1}{d_n}O(d_n)\cdot O(d_n)=O(d_n).
$$
\end{proof}

Assumptions of Proposition~\ref{PropGendDiffable} can already be
used for wider classes of sequences. Let us apply the above
results to a two-parametric family of sequences involving
Example~\ref{Exngamma} as well.

\begin{example} \label{Exngammalneta}
{\em Suppose that the Hamiltonian $\mathrm{H_{X,\alpha}}$ is
generated by the differential expression \eqref{EqlXalpha} as
explained in Introduction, with  $X=\{x_n\}_{0}^{\infty}$ defined
by the relations:
\begin{align}
    & x_0=0,\quad x_n=x_{n-1}+d_n\quad(n\in \mathbb{N}),\notag \\
    &  d_1>0,\quad  d_n=\dfrac{1}{n^{\gamma}\ln^\eta n}
       \quad(n\in\overline{2,\infty}),\label{EqExngammalneta1}\\
    &  \langle\gamma,\eta\rangle\in
      (1/2,\,1)\!\times\!(-\infty,+\infty)\:\cup\:\{1\}\!\times\!(-\infty,0].\notag
\end{align}

Then $\mathrm{H}_{X,\alpha}$ is self-adjoint provided that there holds one of
the following conditions:
\begin{enumerate}
   \item[\textbf{(sa1)}]\;$\sum_{n=1}^{\infty}|\alpha_n|n^{-3\gamma}\ln^{-3\eta}n=\infty\,$.
   \item[\textbf{(sa2)}]\;There exists a constant $C_1>0$ such that
   $$
                           \alpha_n\leq
                           -2\big(n^{\gamma}\ln^\eta n+(n+1)^{\gamma}\ln^\eta(n+1)\big)+
                           \frac{C_1}{n^{\gamma}\ln^\eta n}
                           \quad(n\in \mathbb{N}).
                          $$\label{PageExngammalneta}
   \item[\textbf{(sa3)}]\;There exists a constant $C_2>0$ such that
                          $\alpha_n\geq -\frac{C_2}{n^{\gamma}\ln^\eta n}\quad(n\in \mathbb{N})$.
\end{enumerate}}

Indeed, in this case we have $\{d_n\}_1^{\infty}\in
\ell_2\setminus \ell_1$ and, by direct calculations, we conclude
that $\frac{d_{n+1}}{d_n}=1+O(d_n)$. It is important to note that
the sequence $\frac{d_{n+1}}{d_n}$ doesn't satisfy
estimate~\eqref{asympt} if $\eta\neq 0$. To prove that
\textbf{(sa1)} can be applied, we use Proposition~\ref{Propdn3}.
To prove that \textbf{(sa2)} and \textbf{(sa3)} are applicable,
consider the function $d(x)=\dfrac{1}{x^{\gamma}\ln^\eta x}$
generating the sequence $\{d_n\}_1^\infty$. Derivatives of this
function are of the form:
\begin{align*}
&d'(x)=-(\gamma \ln x+\eta)x^{-\gamma-1}\ln^{-\eta-1}x;\\
&d''(x)=\big(\gamma(\gamma+1)\ln^2x+(2\gamma+1)\eta\ln
x+\eta(\eta+1)\big)x^{-\gamma-2}\ln^{-\eta-2}x.
\end{align*}
This immediately implies conditions \textbf{(d1)}, \textbf{(d2)}
of Proposition~\ref{PropGendDiffable}.

The relation
\begin{align*}
\frac{d''(n)}{d'(n)d_n}=-\frac{\big(\gamma(\gamma+1)\ln^2n+(2\gamma+1)\eta\ln
n+\eta(\eta+1)\big)}{(\gamma \ln^2 n+\eta\ln n)}\cdot
n^{\gamma-1}\ln^{\eta}n
\end{align*}
shows that condition \textbf{(d3)} is fulfilled either for
$\frac{1}{2}<\gamma<1,\ \eta\in(-\infty,+\infty),$ or for
$\gamma=1,\,\eta\leq 0$.

\end{example}

Note\label{Pagel2Gap} that in the two-parametric family of
sequences of form \eqref{EqExngammalneta1} lying in
$\ell_2\setminus \ell_1$ there is a "gap"\ consisting of
sequences of the form $\big\{\frac{1}{n\ln^\eta
n}\big\},\;\eta\in(0,1]$. Condition \textbf{(d0)} is violated for
sequences from this "gap"\ (despite the fact that
$\lim_{n\rightarrow\infty}\frac{d_{n+1}}{d_n}=1$). To investigate
these cases, we must know additional properties of the sequence
$\{d_n\}_1^{\infty}$. Below, in
Proposition~\ref{PropGendDiffable2}, we present analytic
conditions on the function $d$ generating the sequence
$\{d_n\}_1^{\infty}$ such that we may choose the function $G$ from
\eqref{EqFLessG} in estimate \textbf{(II)} of
Proposition~\ref{PropFGgen}.

\begin{proposition}\label{PropGendDiffable2}
Let $p\in\mathbb{N}$, and let a sequence $\{d_n\}_1^{\infty}$ be
generated by the function
$d\colon\;d_n:=d(n)\;\:(n\in\overline{p,\infty})$ that is defined
on the interval $(p,\infty)$ and is continuously differentiable on
it. Assume also fulfillment of conditions
(\textbf{d0}), (\textbf{d1}) of Proposition~\ref{PropGendDiffable}
as well as the following condition:
\begin{enumerate}
   \item[\textbf{(d4)}]\; There exists a $k\in\mathbb{N}$ such that
                         $\big|\frac{d'(n)}{d_n}\big|^{k}=O(d_n^2)$.
\end{enumerate}
Then there exist numbers $\{C_{i}\}_{i\in\overline{0,k-1}}$ such that, for
$k$ from \textbf{(d4)}, we have:
\begin{equation}\label{EqPropGendDiffable1}
    F(n)=
    \frac{1}{d_n}\sum_{i=1}^{k-1}C_{i}u(n)^i+\frac{1}{d_{n+1}}\sum_{i=1}^{k-1}C_{i}v(n)^i
    +O(d_n)
    \quad(n\in\overline{p,\infty}),
\end{equation}
where
$ u(n)=\frac{d_{n+1}-d_{n-1}}{d_{n}+d_{n-1}},\
  v(n)=\frac{d_{n}-d_{n+2}}{d_{n+1}+d_{n+2}}\,.
$
\end{proposition}
\begin{proof}[Proof]
Taking into account~\eqref{EqFDef} we obtain:
\begin{equation*}
  F(n)=\frac{1}{d_{n}}\Bigl(\frac{r_n}{r_{n-1}}-1\Bigr)
  +\frac{1}{d_{n+1}}\Bigl(\frac{r_n}{r_{n+1}}-1\Bigr),
  \end{equation*}
where
\begin{align*}
   & \frac{r_n}{r_{n-1}}=\sqrt{\frac{d_{n}+d_{n+1}}{d_{n}+d_{n-1}}}=\sqrt{1+u(n)},\\
   & \frac{r_n}{r_{n+1}}=\sqrt{\frac{d_{n}+d_{n+1}}{d_{n+1}+d_{n+2}}}=\sqrt{1+v(n)}.
\end{align*}
Using Taylor's series expansion of the function
$\sqrt{1+x}=1+\sum_{i=1}^{\infty}C_ix^i$ we have:
\begin{equation}\label{EqFnSer}
  F(n)=\frac{1}{d_{n}}\sum_{i=1}^{\infty}C_iu(n)^i
  +\frac{1}{d_{n+1}}\sum_{i=1}^{\infty}C_iv(n)^i.
  \end{equation}
Let us estimate the behavior of terms of these series at infinity. In view of \textbf{(d0)},
\textbf{(d1)}, for some $\zeta_n\in[-1,1],\ \theta_n\in[0,2]$ we have:
\begin{align*}
  & u(n)=2\frac{d'(n+\zeta_n)}{d_{n}+d_{n-1}}=2\frac{d'(n+\zeta_n)}{d'(n)}\cdot
          \frac{d_n}{d_{n}+d_{n-1}}\cdot \frac{d'(n)}{d_n}=
          O\Big(\frac{|d'(n)|}{d_n}\Big);\\
  & v(n)=-2\frac{d'(n+\theta_n)}{d_{n+1}+d_{n+2}}=-2\frac{d'(n+\theta_n)}{d'(n)}\cdot
           \frac{d_n}{d_{n+1}+d_{n+2}}\cdot \frac{d'(n)}{d_n}=
           O\Big(\frac{|d'(n)|}{d_n}\Big).
  \end{align*}
Then, for $k\in\mathbb{N}$ from \textbf{(d4)}, we obtain:
\begin{equation*}
   \sum_{i=k}^{\infty}C_iu(n)^i=
   O\big(u(n)^k\big)=O\Big(\frac{|d'(n)|^k}{d_n^k}\Big)=O(d_n^2).
\end{equation*}
Similarly, we conclude that $\sum_{i=k}^{\infty}C_iv(n)^i=
O(d_n^2)$. Hence estimate~\eqref{EqPropGendDiffable1} holds.
\end{proof}

In some cases, the right-hand side of \eqref{EqPropGendDiffable1}
can be used to select from $F$ "the best"\ estimator $G$
satisfying the relation $F(n)=G(n)+O(d_n)$. Such a function $G$
should include all the "parts"\ of sequences
$\tfrac{d'(n)^i}{d_n^{i+1}}\;\:(i\in\overline{1,k-1})$ that grow
slower than $d_n$ at infinity. Note that, to obtain more effective
estimates in assumptions of Proposition~\ref{PropGendDiffable2},
the following argument is useful. Since
\begin{align*}
   & \frac{1}{d_n}u(n)-\frac{1}{d_{n+1}}v(n)
     = \frac{1}{d_n}\cdot
       \frac{d'(n+\zeta)}{d_{n}+d_{n-1}}-
       \frac{1}{d_{n+1}}\cdot\frac{d'(n+\theta)}{d_{n+1}+d_{n+2}}=\\
   & =\frac{d_n}{d_n+d_{n+1}}\cdot \frac{d'(n+\zeta)}{d_{n}^2}-
      \frac{d_{n+1}}{d_{n+1}+d_{n+2}}\cdot
      \frac{d^2_n}{d^2_{n+1}}\cdot
      \frac{d'(n+\theta)}{d_{n}^2},
  \end{align*}
and since both expressions $\frac{d_n}{d_n+d_{n+1}}$ and
$\frac{d_{n+1}}{d_{n+1}+d_{n+2}}\cdot
\big(\frac{d_n}{d_{n+1}}\big)^2$ are close to 1 for
sufficiently large $n$, we conclude that the behavior of the summand
$C_1\big(\frac{1}{d_n}u(n)-\frac{1}{d_{n+1}}v(n)\big)$ at infinity is determined by
the expression $\tfrac{d''(n)}{d_n^2}$. Trying to avoid general definitions here
we will carry out the reasoning in the following example.

\begin{example} \label{Exnlnneta}
Consider the sequences $\big\{\frac{1}{n\ln^\eta
n}\big\},\;\eta\in(0,1]$ (see the argument before
Proposition~\ref{PropGendDiffable}, p.\,\pageref{Pagel2Gap}).

Put $d(x)=x^{-1}\ln^{-\eta}x,\;\eta\in(0,1]$, and $p=3$. The function
$d$ derives the sequence $\big\{\frac{1}{n\ln^\eta
n}\big\}$, is defined on the interval $(3,\infty)$ and is twice continuously
differentiable on it. We have:
\begin{align*}
   & d'(x)=-\tfrac{\ln x+\eta}{x^2\ln^{\eta+1} x};\\
   & \tfrac{d'(x)}{d(x)}=-{\tfrac {\ln x+\eta}{x\ln x}},
\end{align*}
from which we see that $\big(\tfrac{d'(x)}{d(x)}\big)^3=O(d_n^2).$
Hence $k=3$, and for estimating the function $F(n)$ we need to
consider the expressions:
\begin{align*}
   & d''(x)=
     \tfrac{1}{x^3}\big(\tfrac{2}{\ln^\eta x}+
                         \tfrac{3\eta-1}{\ln^{\eta+1} x}+
                         \tfrac{\eta+\eta^2}{\ln^{\eta+2} x}
                   \big);\\
   & \tfrac{d''(x)}{d^2(x)}=
     \tfrac{2\ln^\eta x}{x}+
     \tfrac{3\eta-1}{x\ln^{1-\eta}x}+
     \tfrac{\eta+\eta^2}{x\ln^{2-\eta} x}\,;\\
   & \tfrac{d'(x)^2}{d^3(x)}=
     \tfrac{\ln^{\eta} x}{x}+\tfrac{2\eta}{x\ln^{1-\eta}x}+
     \tfrac{\eta^2}{x\ln^{2-\eta}x}\,.
\end{align*}
This makes us possible to presuppose that $F$ admits the asymptotic
representation
$$
  F(n)=w_1\tfrac{\ln^\eta n}{n}+w_2\tfrac{1}{n\ln^{1-\eta}n}+O\big(\tfrac{1}{n\ln^\eta n}\big)
$$
with some coefficients $w_1,\ w_2$. In fact, by direct calculations
we can obtain that
\begin{align*}
   & \underset{n\rightarrow\infty}{\lim}\tfrac{n}{\ln^\eta n}F(n)=\tfrac{1}{4};\\
   & \underset{n\rightarrow\infty}{\lim}n\ln^{1-\eta}n\big(F(n)-\tfrac{1}{4}\tfrac{\ln^\eta n}{n}\big)=\eta;\\
   & \underset{n\rightarrow\infty}{\lim}
     n\ln^\eta n\big(F(n)-\tfrac{1}{4}\tfrac{\ln^\eta n}{n}-\tfrac{\eta}{n\ln^{1-\eta}n}\big)=
     \left\{\!\!
       \begin{array}{ll}
         0,            & \eta\in(0,1) \\
         \tfrac{1}{4}, & \eta=1
       \end{array}
     \right.,
\end{align*}
for $\eta\in(0,1]$, which gives the asymptotics:
\begin{equation}\label{EqExnlnneta2}
  F(n)=
  \Bigg\{\!\!\!
    \raisebox{0.007\textwidth}
    {
    \parbox{0.6\textwidth}
    {$
    \begin{array}{ll}
       \lefteqn{\phantom{\frac{1}{2}}}
       \tfrac{1}{4}\tfrac{\ln^\eta n}{n}+
       O\big(\tfrac{1}{n\ln^\eta n}\big),    & \eta\in(0,\tfrac{1}{2}] \\
       \tfrac{1}{4}\tfrac{\ln^\eta n}{n}+
      \tfrac{\eta}{n\ln^{1-\eta}n}+
       O\big(\tfrac{1}{n\ln^\eta n}\big),    & \eta\in (\tfrac{1}{2},1]
    \end{array}
    $}
    }
  \Bigg.
\end{equation}

Finally, we have:
$
  G(n)=
  \Bigg\{\!\!\!
    \raisebox{0.007\textwidth}
    {
    \parbox{0.32\textwidth}
    {$
    \begin{array}{ll}
       \lefteqn{\phantom{\frac{1}{2}}}
       \tfrac{1}{4}\tfrac{\ln^\eta n}{n},     & \eta\in(0,\tfrac{1}{2}]\\
       \tfrac{1}{4}\tfrac{\ln^\eta n}{n}+
       \tfrac{\eta}{n\ln^{1-\eta}n},         & \eta\in (\tfrac{1}{2},1]
    \end{array}
    $}
    }
  \Bigg.
\qquad .$ Let us summarize our considerations in the form of
sufficient conditions for the operator $\mathrm{H_{X,\alpha}}$ to
be self-adjoint.

\medskip
{\em Suppose that the Hamiltonian $\mathrm{H_{X,\alpha}}$ is
generated by the differential expression \eqref{EqlXalpha} as
explained in Introduction, with  $X=\{x_n\}_{0}^{\infty}$ defined
by the relations:
\begin{align*}
   & x_0=0,\quad x_n=x_{n-1}+d_n\quad(n\in \mathbb{N}),\\
   &  d_1>0,\quad  d_n=\tfrac{1}{n\ln^\eta n}
     \quad(n\in\overline{2,\infty}\,),\quad \eta\in(0,1].
\end{align*}
Then $\mathrm{H}_{X,\alpha}$ is self-adjoint provided that
there holds one of the following conditions:
\begin{enumerate}
   \item[\textbf{(sa1)}]\;$\sum_{n=1}^{\infty}|\alpha_n|n^{-3}\ln^{-3\eta}n=\infty\,$.
   \item[\textbf{(sa2)}]\;There exists a constant $C_1>0$ such that $$
                           \alpha_n\leq
                           {\scriptstyle
                           -2\big(n\ln^\eta n+(n+1)\ln^\eta(n+1)\big)
                           }+
                           \Bigg\{\!\!\!
                               \raisebox{0.007\textwidth}
                               {
                               \parbox{0.29\textwidth}
                               {$
                               \begin{array}{ll}
                                  \lefteqn{\phantom{\frac{1}{2}}}
                                  \tfrac{1}{4}\tfrac{\ln^\eta n}{n},    & {\scriptstyle\eta\in(0,\tfrac{1}{2}]}\\
                                  \tfrac{1}{4}\tfrac{\ln^\eta n}{n}+
                                  \tfrac{\eta}{n\ln^{1-\eta}n},         & {\scriptstyle\eta\in (\tfrac{1}{2},1]}
                               \end{array}
                               $}
                               }
                           \Bigg.\qquad
                           +\:\tfrac{C_1}{n\ln^\eta n}.
                          $$
   \item[\textbf{(sa3)}]\; There exists a constant $C_2>0$ such that
                          $$
                           \alpha_n\geq
                           \Bigg\{\!\!\!
                               \raisebox{0.007\textwidth}
                               {
                               \parbox{0.294\textwidth}
                               {$
                               \begin{array}{ll}
                                  \lefteqn{\phantom{\frac{1}{2}}}
                                  \tfrac{1}{4}\tfrac{\ln^\eta n}{n},    & {\scriptstyle\eta\in(0,\tfrac{1}{2}]}\\
                                  \tfrac{1}{4}\tfrac{\ln^\eta n}{n}+
                                  \tfrac{\eta}{n\ln^{1-\eta}n},         & {\scriptstyle\eta\in (\tfrac{1}{2},1]}
                               \end{array}
                               $}
                               }
                           \Bigg.\qquad
                           -\,\tfrac{C_2}{n\ln^\eta n}.
                          $$
\end{enumerate}
}
\end{example}

\section{Sufficient conditions for non-triviality of $n_{\pm}(\mathrm{H_{X, \alpha}})$}

For a positive sequence $\{d_n\}_1^{\infty}$, define the sequence
$\{\tilde{r}_n\}_1^{\infty}$ recursively:
\begin{equation}\label{EqrntildeRecur}
  \tilde{r}_1:=1,\quad
  \tilde{r}_{n+1}:=-\frac{d_{n+1}}{\tilde{r}_{n}}\quad(n\in \mathbb{N})
\end{equation}
(here we generalize argument from \cite[Proposition
5.13]{karpiiKost} regarding the case $d_n=1/n$).

It is easy to show by induction that
\begin{equation}\label{Eqrntilde}
   \tilde{r}_{n+1}:=(-1)^n\frac{d_{n+1}d_{n-1}\ldots}{d_{n}d_{n-2}\ldots}
   \quad(n\in \mathbb{N}).
\end{equation}

We say that a sequence $\{d_n\}_1^{\infty}$ satisfies condition $(A)$ if
$$
\{r_n\widetilde{r}_{n}\}_1^{\infty}\in \ell_2,
$$
and we say that it satisfies condition $(B)$ if
$$
  \Bigl(\frac{1}{d_n}+\frac{1}{d_{n+1}}\Bigr)\
  \tilde{r}_{n}^{2}=u_n+O(r_n^{2}\tilde{r}_{n}^{2}),
$$
where $\{u_n\}_1^\infty$ stands for a real \textit{periodic} sequence.

It follows from the results of~\cite{karpiiKost} that the
sequence $d_n=1/n$ satisfies both conditions $(A)$ and $(B)$, and
the period of the sequence $\{u_n\}_1^\infty$ equals 2;
$u_1=4/\pi,\ u_2=\pi.$ As it turns out, this result is general
enough for sequences satisfying conditions $(A)$ and $(B)$.
Namely, the following statement holds.

\begin{lemma}\label{LemPerNper2}
Suppose that a sequence $\{d_n\}_1^{\infty}\in \ell_2\backslash
\ell_1$ such that
$\lim_{n\rightarrow\infty}\frac{d_{n+1}}{d_n}=1$ satisfies also conditions
$(A)$ and $(B)$. Then the sequence $\{u_n\}_1^\infty$
has the period equal to 2, with $u_1u_2=4$.
\end{lemma}

\begin{proof}[Proof]
Let the sequence $\{d_n\}_1^{\infty}$ satisfy conditions $(A)$ and $(B)$,
and let $N$ be the period of the sequence $\{u_n\}_1^\infty$. We denote
$$
  \rho_n:= \Bigl(\frac{1}{d_n}+\frac{1}{d_{n+1}}\Bigr)\
  \tilde{r}_{n}^{2}.
$$
In view of condition $(B)$,
$\rho_{1+kN}=u_1+O(r_{1+kN}^{2}\tilde{r}_{1+kN}^{2}).$ If
$1<s<N$ is an arbitrary odd number, we have:
\begin{equation}\label{EqLemPerNper21}
   \rho_{s+kN}=\Bigl(\frac{1}{d_{s+kN}}+\frac{1}{d_{s+kN+1}}\Bigr)\
               \Bigl(\frac{d_{s+kN}d_{s+kN-2}\ldots}{d_{s+kN-1}
               d_{s+kN-3}\ldots}\Bigr)^2=\Theta(s,k)\rho_{1+kN},
\end{equation}
where
$$
  \Theta(s,k)=\Bigl(\frac{1}{d_{s+kN}}+\frac{1}{d_{s+kN+1}}\Bigr)\
              \Bigl(\frac{1}{d_{1+kN}}+\frac{1}{d_{2+kN}}\Bigr)^{-1}\Bigl(\frac{d_{s+kN}
              d_{s+kN-2}\ldots d_{3+kN}}{d_{s+kN-1}
              d_{s+kN-3}\ldots d_{2+kN}}\Bigr)^2,
$$
and also $\lim_{k\rightarrow \infty} \Theta(s,k)=1.$ Since
$\rho_{s+kN}=u_s+O(r_{s+kN}^{2}\tilde{r}_{s+kN}^{2})$, then, by passing to the limit
in relation~\eqref{EqLemPerNper21} at $k\rightarrow \infty$, we obtain
$$
  u_s=u_1.
$$
Thus, for an arbitrary odd $n=2k+1$ we have:
$$\rho_{2k+1}=u_1+O(r_{2k+1}^{2}\tilde{r}_{2k+1}^{2}).$$

Arguing similarly we can show that for an arbitrary even $n=2k $ we obtain
$$
  \rho_{2k }=u_2+O(r_{2k}^{2}\tilde{r}_{2k}^{2}).
$$
Since
$$
  \rho_{2k}\rho_{2k+1}=
  \Bigl(1+\frac{d_{2k+1}}{d_{2k+2}}\Bigr)\Bigl(1+\frac{d_{2k+1}}{d_{2k}}\Bigr),
$$
we arrive at the relation
$$
  \bigl(u_1+O(r_{2k+1}^{2}\tilde{r}_{2k+1}^{2})\bigr)\bigl(u_2+O(r_{2k}^{2}\tilde{r}_{2k}^{2})\bigr)=
  \Bigl(1+\frac{d_{2k+1}}{d_{2k+2}}\Bigr)\Bigl(1+\frac{d_{2k+1}}{d_{2k}}\Bigr).
$$
Finally, by passing to the limit in the latter at $k\rightarrow \infty$, we
conclude that
$$
  u_1u_2=4.
$$
\end{proof}

\begin{proposition}\label{PropHamSym}
Suppose that the Hamiltonian $\mathrm{H}_{X,\alpha}$ is defined by the sequence
$\{d_n\}_1^{\infty}\in \ell_2\backslash \ell_1$ such that
$\lim_{n\rightarrow\infty}\frac{d_{n+1}}{d_n}=1$ and for which conditions
$(A)$ and $(B)$ are satisfied.

If
$$
\alpha_n=a\Bigl(\frac{1}{d_n}+\frac{1}{d_{n+1}}\Bigr)+O(d_n),
$$
where the parameter $a$ satisfies the inequality $-2<a<0$, then the
Hamiltonian $\mathrm{H}_{X,\alpha}$ is the symmetric operator
with deficiency indices $n_{\pm}=1.$
\end{proposition}

\begin{proof}[Proof]
Due to Lemma~\ref{LemPerNper2}, for the given sequence $d_n$,
the real periodic sequence $\{u_n\}_1^\infty$ in condition $(B)$
has the period equal to 2, with $u_1u_2=4$.

Consider the sequence
$$
\alpha_n^0:=-\Bigl(\frac{1}{d_n}+\frac{1}{d_{n+1}}\Bigr)+(a+1)u_n\tilde{r}_n^{-2}.
$$
It follows that
$$
B_{X,\alpha^0}=\left(%
\begin{array}{cccc}
 r_1^{-2}\tilde{r}_1^{-2} (a+1)u_1 & -r_1^{-1}r_2^{-1}d_2^{-1}    & 0    & \ldots  \\
 -r_1^{-1}r_2^{-1}d_2^{-1}   & r_2^{-2}\tilde{r}_2^{-2}(a+1)u_2 & -r_2^{-1}r_3^{-1}d_3^{-1}    & \ldots  \\
  0    & -r_2^{-1}r_3^{-1}d_3^{-1}    & r_3^{-2}\tilde{r}_3^{-2}(a+1)u_1 & \ldots \\
  \ldots  & \ldots  & \ldots  & \ldots \\
\end{array}%
\right).
$$
If $R_X=\mathrm{diag}(r_n),\
\widetilde{R}_1=\mathrm{diag}(\widetilde{r}_n),$ we have
$$
\widetilde{R}_1R_XB_{X,\alpha^0}R_X\widetilde{R}_1=\left(%
\begin{array}{cccc}
 (a+1)u_1 & -\tilde{r}_1\tilde{r}_2d_2^{-1}    & 0    & \ldots  \\
 -\tilde{r}_1\tilde{r}_2d_2^{-1}   & (a+1)u_2 & -\tilde{r}_2\tilde{r}_3d_3^{-1}    & \ldots  \\
  0    & -\tilde{r}_2\tilde{r}_3d_3^{-1}    & (a+1)u_1 & \ldots \\
  \ldots  & \ldots  & \ldots  & \ldots \\
\end{array}%
\right).
$$
Since $-\tilde{r}_n\tilde{r}_{n+1}d_{n+1}^{-1}=
-\tilde{r}_n\cdot \frac{-d_{n+1}}{\tilde{r}_n}\cdot
\frac{1}{d_{n+1}}=1,$ then
\begin{equation}\label{EqPropHamSym1}
\widetilde{R}_1R_XB_{X,\alpha^0}R_X\widetilde{R}_1=J_a,
\end{equation}
where
$$
J_a=\left(%
\begin{array}{cccc}
  (a+1)u_1 & 1    & 0    & \ldots  \\
  1    & (a+1)u_2 & 1    & \ldots  \\
  0    & 1    & (a+1)u_1 & \ldots  \\
  \ldots  & \ldots  & \ldots  & \ldots \\
\end{array}%
\right)
$$
is the periodic Jacobi matrix defined by both the sequence
$\{u_n\}_1^\infty$ and the real parameter $a.$

In view of~\eqref{EqPropHamSym1}, the vector $f^0$ is the solution of the equation
$J_af=0$ if and only if the vector $h^0=
R_X\widetilde{R}_1f^0$ is the solution of the equation $B_{X,\alpha^0}h=0.$
If $f^0=\{f_n\}$ is a bounded sequence, due to condition $(A)$
we obtain that $h^0=\{r_n\tilde{r}_nf_n\}\in \ell_2.$

As is known~\cite{karpiiTes}, solutions to the equation $J_af=0$
are bounded if there holds the inequality $|\Delta_a(0)|<1$ for
the Floquet discriminant. It follows from the above considerations
that the Jacobi matrix determined by the sequence
$\{u_n\}_1^\infty$ has the period equal to 2. Consequently,
$$
  \Delta_a(\lambda)
  =1/2\bigl(-2+(\lambda-(a+1)u_1)(\lambda-(a+1)u_2)\bigr).
$$
This yields
$\Delta_a(0)=1/2(-2+(a+1)^2u_1u_2)=1/2(-2+4(a+1)^2)=2(a+1)^2-1.$
Hence, $|\Delta_a(0)|< 1$ if $-2<a<0.$ Thus, under this
condition, a solution to the equation $B_{X,\alpha^0}h=0$ belongs
to $\ell_2$, and also $B_{X,\alpha^0}$ is symmeric operator with
deficiency indices $n_{\pm}=1.$

We can simplify the general form of the sequence $\alpha^0$.
Indeed, condition $(B)$ implies that
$$
(a+1)u_n\tilde{r}_n^{-2}=(a+1)\Bigl(\frac{1}{d_n}+\frac{1}{d_{n+1}}\Bigr)+
(a+1)O(r_n^{2}\tilde{r}_{n}^{2})\tilde{r}_n^{-2},
$$
where
$(a+1)O(r_n^{2}\tilde{r}_{n}^{2})\tilde{r}_n^{-2}=O(r_n^{2})=O(d_n).$
It follows that
$$
(a+1)u_n\tilde{r}_n^{-2}=(a+1)\Bigl(\frac{1}{d_n}+\frac{1}{d_{n+1}}\Bigr)+
O(d_n),
$$
and the sequence $\alpha$ of the form
$$
  \alpha_n=a\Bigl(\frac{1}{d_n}+\frac{1}{d_{n+1}}\Bigr)+ O(d_n)
$$
derives a bounded self-adjoint perturbation of the operator
$B_{X,\alpha^0}.$ Since deficiency indices don't change under such
a perturbation, we conclude that $B_{X,\alpha}$ is symmetric
operator with deficiency indices $n_{\pm}=1$ as well.
\end{proof}

\begin{example}\label{Ex1ngamma}
Let $d_n=\tfrac{1}{n^{\gamma}}\;\:(n\in\mathbb{N}),\
\gamma\in(\tfrac{1}{2},1]$. For the given sequence, we have:
$$
  \tilde{r}_{n}=(-1)^{n-1}\big(\tfrac{(n-1)!!}{n!!}\big)^{\gamma}
  \quad(n\in\mathbb{N}),
$$
which implies that
\begin{equation}\label{EqEx1ngamma1}
  \tilde{r}_{n}^2=
  \big(\tfrac{(n-1)!!}{n!!}\big)^{2\gamma}=
  \tfrac{1}{(2n+1)^{\gamma}}
  \big((2n+1)\big(\tfrac{(n-1)!!}{n!!}\big)^2\big)^{\gamma}
  \quad(n\in\mathbb{N}).
\end{equation}

Let us use the asymptotics
\begin{equation}\label{EqEx1ngamma2}
   (2n+1)\big(\tfrac{(n-1)!!}{n!!}\big)^2=
   \bigg\{\!\!\!
    \raisebox{0.005\textwidth}
    {
    \parbox{0.294\textwidth}
    {$
    \begin{array}{ll}
       \pi + O(n^{-2}),            & \text{if}\  n\  \text{is \textit{odd}};   \\
       \tfrac{4}{\pi} + O(n^{-2}), & \text{if}\  n\  \text{is \textit{even}}
    \end{array}
    $}
    }
  \bigg.
\end{equation}
obtained in \cite[Proposition 5.13]{karpiiKost}, and also the
asymptotics derived by the chain
\begin{equation}\label{EqEx1ngamma3}
   \tfrac{n^{\gamma}+(n+1)^{\gamma}}{(2n+1)^{\gamma}}=
   \tfrac{1+\big(1+\tfrac{1}{n}\big)^\gamma}
         { \strut
           2^\gamma\big(1+\tfrac{1}{2n}\big)^\gamma
         }=
   \tfrac{1}{2^\gamma}
   \big(2+\tfrac{\gamma}{n}+O(n^{-2})\big)
   \big(1-\tfrac{\gamma}{2n}+O(n^{-2})\big)=
   2^{1-\gamma}+O(n^{-2}).
\end{equation}

We have several relations ($n\in\mathbb{N}$):
\begin{equation*}
   d_n=n^{-\gamma}\sim r_n^2\,;\qquad
   n^{-2}=O(n^{-2\gamma})\,;\qquad
   \tilde{r}_{n}^{2}
   \overset{(\ref{EqEx1ngamma1},\,\ref{EqEx1ngamma2})}{=}
   O(n^{-\gamma})\,.
\end{equation*}
They immediately yield that
\begin{equation}\label{EqEx1ngamma4}
   \text{a)}\;\:
   n^{-2}=O(d_n\tilde{r}_{n}^2)\,;\qquad
   \text{b)}\;\:
   r_{n}^2\tilde{r}_{n}^2=O(n^{-2\gamma})
\end{equation}
(here $n\in\mathbb{N}$). In view of (\ref{EqEx1ngamma4}b), condition
$(A)$ holds for the sequence $\{d_n\}_1^{\infty}$. Moreover,
putting $w:=(\pi,\,\tfrac{4}{\pi},\,\pi,\,\tfrac{4}{\pi},\ldots)$
and $u_n:=2^{1-\gamma}w_n^\gamma\;\;(n\in\mathbb{N})$, we obtain
the chain:
\begin{equation*}
  \begin{split}
      \big(\tfrac{1}{d_n}+\tfrac{1}{d_{n+1}}\big)\tilde{r}_{n}^{2}
    & \overset{\eqref{EqEx1ngamma1}}{=}
      \tfrac{n^{\gamma}+(n+1)^{\gamma}}{(2n+1)^{\gamma}}
      \big((2n+1)\big(\tfrac{(n-1)!!}{n!!}\big)^2\big)^\gamma
      \overset{(\ref{EqEx1ngamma2},\,\ref{EqEx1ngamma3})}{=}\\
    & \overset{\phantom{\eqref{EqEx1ngamma1}}}{=}
      \big(2^{1-\gamma}+O(n^{-2})\big)
      \big(w_n+O(n^{-2})\big)^\gamma
      \overset{(\ref{EqEx1ngamma4}a)}{=}
      u_n+O(d_n\tilde{r}_{n}^2)
      \quad(n\in\mathbb{N}).
  \end{split}
\end{equation*}
We combine the essay of these considerations with particular case
\textbf{(sa2)}, \textbf{(sa3)} of Example~\ref{Exngammalneta} for
$\eta=0$ (see p.\,\pageref{PageExngammalneta}) in order to
demonstrate the dependence of the Hamiltonian
$\mathrm{H}_{X,\alpha}$ on an asymptotic behavior of the sequence
$\alpha$.

\medskip
{\em Suppose that the Hamiltonian $\mathrm{H_{X,\alpha}}$ is
generated by the differential expression \eqref{EqlXalpha} as
explained in Introduction, with  $X=\{x_n\}_{0}^{\infty}$ defined
by the relations:
\begin{equation*}
   x_0=0,\quad x_n=x_{n-1}+\tfrac{1}{n^\gamma}
   \quad(n\in \mathbb{N}),
   \quad\gamma\in(\tfrac{1}{2},1].
\end{equation*}
Then we have:

if\;\:$
      \alpha_n\leq
     -2\big(n^{\gamma}+(n+1)^{\gamma}\big)+
      C_1n^{-\gamma}
      \;\:(n\in\mathbb{N})
      $%
\ for some $C_1>0,$\\
\phantom{\kern\parindent{if}\;\:%
         $\alpha_n\geq -C_2n^{-\gamma}\;\:(n\in\mathbb{N})$ for some $C_2>0$,%
        }%
\;then $\mathrm{H_{X,\alpha}}$ is self-adjoint;

if\;\:$
       \alpha_n =
       a\big(n^{\gamma}+(n+1)^{\gamma}\big)+O(n^{-\gamma})
       \;\:(n\in\mathbb{N})
      $%
\ for some $a\in(-2,0)$,\\
\phantom{\kern\parindent{if}\;\:%
         $\alpha_n\geq -C_2n^{-\gamma}\;\:(n\in\mathbb{N})$ for some $C_2>0$,%
        }%
\;then $n_{\pm}(\mathrm{H_{X,\alpha}})=1$;

if\;\:$\alpha_n\geq -C_2n^{-\gamma}\;\:(n\in\mathbb{N})$ for some $C_2>0$,%
\;then $\mathrm{H_{X,\alpha}}$ is self-adjoint. }
\end{example}

\medskip

We are grateful to M.\,M.~Malamud and A.\,S.~Kostenko for the statement of problem and useful advices.

\medskip

\renewcommand{\refname}{References}

 \end{document}